\documentclass[11pt]{amsart}
\usepackage{amsmath,amssymb,amsthm}
\usepackage{verbatim}
\usepackage[english]{babel}
\usepackage[utf8]{inputenc}
\usepackage{csquotes}
\usepackage{tikz-cd}
\usepackage{cleveref}
\usepackage{enumerate}
\usepackage{graphicx}
\usepackage{mathtools}
\usepackage[all,cmtip]{xy}

\usepackage[doi=false,isbn=false,url=false,giveninits=true,maxnames=99,maxalphanames=5,style=alphabetic]{biblatex}
\AtEveryBibitem{\ifentrytype{article}{
    \clearfield{number}
}{}
}
\renewbibmacro{in:}{\ifentrytype{article}{}{\printtext{\bibstring{in}\intitlepunct}}}
\AtEveryBibitem{\ifentrytype{book}{
    \clearfield{pages}
}{}
}

\newtheorem{theorem}{Theorem}[section]
\newtheorem{conjecture}[theorem]{Conjecture}
\newtheorem{lemma}[theorem]{Lemma}
\newtheorem{corollary}[theorem]{Corollary}
\newtheorem{proposition}[theorem]{Proposition}

\theoremstyle{definition}
\newtheorem{definition}[theorem]{Definition}

\newtheorem{remark}[theorem]{Remark}
\newtheorem*{remark*}{Remark}
\newtheorem*{hypothesis*}{Setting}

\usepackage{geometry}
\usepackage{color}

\DeclareMathOperator{\Alb}{Alb}
\DeclareMathOperator{\Aut}{Aut}

\DeclareMathOperator{\End}{End}

\DeclareMathOperator{\id}{id}

\DeclareMathOperator{\Jac}{Jac}
\DeclareMathOperator{\Ker}{Ker}
\DeclareMathOperator{\Kum}{Kum}
\DeclareMathOperator{\lcm}{lcm}

\DeclareMathOperator{\Sym}{Sym}

\newcommand{\ov}{\overline}
\newcommand{\wt}{\widetilde}

\newcommand{\BN}{{\mathbb {N}}}

\newcommand{\BP}{{\mathbb {P}}}

\newcommand{\BZ}{{\mathbb {Z}}}

\newcommand{\CC}{{\mathcal {C}}}

\newcommand{\CN}{{\mathcal {N}}}

\newcommand{\CP}{{\mathcal {P}}}

\newcommand{\RH}{{\mathrm {H}}}

\newcommand{\tX}{{\widetilde{X}}}

\title[Rational curves and the Hilbert Property]{Rational curves and the Hilbert Property on Jacobian Kummer varieties}
\author{Damián Gvirtz-Chen \and Zhizhong Huang}
\address{School of Mathematics and Statistics,
University of Glasgow,
University Place,
Glasgow G12 8QQ
United Kingdom}
\email{damian.gvirtz@glasgow.ac.uk}
\address{Institute of Mathematics, Academy of Mathematics and Systems Science, Chinese Academy of Sciences, Beijing, 100190, China}
\email{zhizhong.huang@yahoo.com}
\begin{document}
    \begin{abstract}
     Let $K$ be a finitely generated field of characteristic zero. A conjecture by Corvaja and Zannier predicts that smooth, projective, simply connected varieties over $K$ with Zariski dense set of rational points have the Hilbert Property.
    
    In this article, we confirm this conjecture for all Kummer surfaces associated to the Jacobian of a genus $2$ curve over $K$, and in general for all Jacobian Kummer varieties associated to a hyperelliptic curve of genus $\geq 3$ of odd degree over $K$.
	\end{abstract}
\maketitle
	
\section{Introduction}
The aim of this article is to introduce the use of rational curves into the study of the Hilbert Property for non-unirational varieties.

\subsection{Background for the Hilbert Property}
Let $X$ be a quasi-projective integral variety of dimension at least one over a field $K$. Recall after \cite[\S3]{Serre} (see also \cite[\S2]{C-T2020AA}) that a subset $A\subset X(K)$ is \emph{thin} if there exists a $K$-morphism $f:Y\to X$ from a $K$-variety $Y$ to $X$ such that \begin{itemize}
    \item $A\subset f(Y(K))$;
    \item the fibre of $f$ over the generic point of $X$ is finite, and $f$ does not have any rational sections.
\end{itemize}
\begin{definition}[\cite{CT-Sansuc} p. 189]\label{def:HP}
We say that $X$ has the \emph{Hilbert Property} (or $X$ is \emph{of Hilbert type}) if $X(K)$ is not a thin set.
\end{definition}
The field $K$ is called \emph{Hilbertian} if the projective line $\BP^1_K$ satisfies the Hilbert Property.  Not every base field is Hilbertian. The field of rational numbers, and more generally all number fields, are classically known to be Hilbertian by the Hilbert Irreducibility Theorem \cite{Hilbert} (see \cite[Theorem~3.4.3]{Serre}). In this article, unless explicitly stated to the contrary, we will assume that $K$ is a finitely generated field of characteristic $0$. These fields are known to be Hilbertian (see e.g. \cite[Chapter 9]{Lang} and \cite[\S13.4]{Fried-Jarden}).

It should be pointed out that \Cref{def:HP} looks different from the one introduced later in \cite[\S1.1 Definition]{Corvaja-Zannier}, \cite[Definition 1.1]{CDJLZ}. We shall provide a proof (see \Cref{prop:thincoverequiv} and \Cref{co:thincoverequiv} below) that all these definitions are actually equivalent.

It is clear from the definition that Zariski density of $X(K)$ in $X$ is a necessary condition for the Hilbert Property, and that the Hilbert Property is a birational invariant. Moreover, if $X$ satisfies the Hilbert Property, then it is geometrically integral (see \cite[p. 20]{Serre}). 
The Hilbert Property is a specialisation of weak weak approximation (see \cite[\S3.5, Corollary 3.5.4, Definition 3.5.6]{Serre} and also \cite[\S2 (AFF), Th\'eo\-r\`eme 2.3]{C-T2020AA}) and provides an approach to solving the Inverse Galois Problem (see for example the book \cite{Serre} and notably \textit{ibid.}, Theorem 3.5.9).   

There is a well-known topological obstruction to the Hilbert Property which goes back to the Chevalley--Weil Theorem. Namely, a projective variety $X$ cannot have the Hilbert Property if it is not algebraically simply connected (or equivalently for normal $X$, if its fundamental group has a proper subgroup of finite index), see \cite[Theorem~1.4]{Corvaja-Zannier}. Corvaja and Zannier conjecture that the converse is true.
\begin{conjecture}[\cite{Corvaja-Zannier}, \S2 Question-Conjecture~1]\label{conj:CZ}
Let $X$ be a smooth, projective and algebraically simply connected variety over $K$. If $X(K)$ is Zariski dense, then $X$ has the Hilbert Property.\footnote{In \cite[\S2 Question-Conjecture~1bis]{Corvaja-Zannier} one finds an affine analogue of this conjecture concerning integral points. However, we are informed by Prof.\ Corvaja that in some cases of the analogue a finite field extension is necessary. Whether this is also the case for \Cref{conj:CZ} is unclear to the authors.}
\end{conjecture}
While there have been quite a number of results confirming the weak weak approximation property for unirational varieties (they are simply connected by \cite{Serre-uni}), as conjectured by Colliot-Th\'el\`ene (see \cite{CT-Sansuc}, \cite[Conjecture 3.5.8]{Serre}), progress towards the Hilbert Property for intermediate type varieties has been very slow.
In particular Corvaja and Zannier ask for a proof that the Hilbert Property holds for Kummer surfaces with a Zariski dense set of rational points \cite[Appendix~1]{Corvaja-Zannier}. (See \cite[p.\ 817]{H-Sk} for a stronger conjecture of weak weak approximation on K3 surfaces and \cite{Holmes-Pannekoek} for a consequence on the rank of quadratic twists of abelian varieties.)

\subsection{Main results and discussion}
 Let $H$ be a smooth, projective hyperelliptic curve of genus $g\geq 2$ over a finitely generated field $K$ of characteristic $0$. Let $\Jac(H)$ be its Jacobian variety and let $\Kum(\Jac(H))$ be the associated Kummer quotient, i.e.\ the quotient of $\Jac(H)$ by the natural involution. The minimal desingularisation of $\Kum(\Jac(H))$ is a simply connected variety of dimension $g$, called a \emph{Jacobian Kummer variety} (see \S\ref{jacobian} for details).

In dimension two, all principally polarised abelian surfaces over $K$ are isomorphic to either the Jacobian of a genus $2$ curve or, up to a quadratic extension, the product of two elliptic curves (see \cite[Theorem 3.1]{Rotger}). The Jacobian case is the generic one, in the sense of forming a Zariski dense open subvariety of the corresponding moduli space. Our main theorem confirms \Cref{conj:CZ} in the Jacobian case. 
\begin{theorem}\label{dim2}
Let $H$ be a smooth, projective genus $2$ curve over $K$.  Then the variety $\Kum(\Jac(H))$ has the Hilbert Property.\footnote{Note that any smooth genus $2$ curve is hyperelliptic.}
\end{theorem}

The study of the Hilbert Property for K3 surfaces was first addressed by Corvaja and Zannier in \cite{Corvaja-Zannier}, and they verified it for the diagonal quartic surface $x^4+y^4=z^4+w^4$ \cite[Theorem 1.6]{Corvaja-Zannier} using multiple elliptic fibrations. The same method was formalised and applied in the recent work of Demeio \cite[Proposition 4.2.1]{Demeio}.  He establishes the Hilbert Property for the Kummer surface $\Kum(E_1\times E_2)$ associated to the product of two elliptic curves $E_1,E_2$ over $K$, under the assumption that rational points on $\Kum(E_1\times E_2)$ are Zariski dense. (Demeio's theorem is stated over a number field and contains the stronger assumption that the elliptic curves have positive Mordell-Weil ranks. In fact, the proof of \cite[Theorem 1.0.2]{Demeio} works for any finitely generated field $K$ of characteristic $0$ and only assumes that the set of rational points is Zariski dense. This is known to hold unconditionally unless $j(E_1)=j(E_2)\in\{0,1728\}$, thanks to the work of Kuwata--Wang \cite{K-W}.) The second author studied in \cite{Huang}  various Kummer surfaces associated to torsors under the product of two elliptic curves, though the Hilbert Property for such surfaces remains still open. More recently, the first author and Mezzedimi \cite{potential} have shown that the Hilbert Property holds for Kummer surfaces if one allows a finite field extension of the base field $K$.

The elliptic fibrations (modulo automophisms) of a generic Jacobian Kummer surface (i.e. $\Jac(H)$ has no extra endomorphisms) over an algebraically closed field have been classified by Kumar in \cite{Kumar}. Unfortunately, in general none of these fibrations descend to $K$ and thus, it is not possible to apply directly the method of \cite{Demeio} over $K$. Instead, the novelty of our proof of Theorem~\ref{dim2} is, amongst other things, the use of rational curves on $\Kum(\Jac(H))$.
To the authors' knowledge, this is the first time that rational curves have been used to prove the Hilbert Property for non-unirational varieties. 

In fact, our approach allows us to settle \Cref{conj:CZ} for a family of higher-dimensional varieties. 
We are not aware of any previous results on the Hilbert Property for non-unirational varieties of dimension $>2$, apart from naive constructions like product varieties. 
\begin{theorem}\label{higher}
Let $H$ be a smooth, projective hyperelliptic curve of genus $g\geq 3$ and of odd degree over $K$, i.e., $H$ is defined by an affine model $y^2=f(x)$ with $f$ separable of degree $2g+1$. 
Then the variety $\Kum(\Jac(H))$ has the Hilbert Property.
\end{theorem}

\subsection{Strategy and outline}
In order to reduce the \emph{a priori} unknown Hil\-bert Property of a higher-dimensional variety $X$ to the known Hilbert Property of (possibly singular) rational subcurves on $X$, their behaviour under pullbacks along covers of $X$ has to be understood. This is the criterion of \Cref{criterion} which we formulate in a way tailored to our application to Jacobian Kummer varieties.

Let $\Alb^1(H)$ be the (first) Albanese torsor of the hyperelliptic curve $H$, and let $\Kum(\Alb^1(H))$ be the associated Kummer quotient, i.e. the quotient of $\Alb^1(H)$ by the natural involution. See \S\ref{se:alb} for details. In general $\Kum(\Alb^1(H))$  is a non-isomorphic Galois twist of $\Kum(\Jac(H))$.
In practice, it turns out that working with $\Kum(\Alb^1(H))$ instead of $\Kum(\Jac(H))$ is crucial for our purposes. The most general form of our main theorem is the following, from which Theorems \ref{dim2} and \ref{higher} follow as special cases.
\begin{theorem}\label{alb}
Let $H$ be a smooth, projective hyperelliptic curve over $K$.
 Then the variety $\Kum(\Alb^1(H))$ has the Hilbert Property.
\end{theorem}

The variety $\Kum(\Alb^1(H))$ always possesses a countable family of rational curves, whose union is Zariski dense. To show that the conditions in the criterion of \Cref{criterion} are satisfied for $\Kum(\Alb^1(H))$, we prove Pullback Lemmas that deal with the three occurring types of covers of $\Kum(\Alb^1(H))$:
\begin{enumerate}
    \item covers that factor rationally through $\Alb^1(H)$;
    \item covers that arise as Kummer quotients of étale covers of $\Alb^1(H)$;
    \item covers that pull back to ramified covers of $\Alb^1(H)$.
\end{enumerate}
While type (1) covers are relatively easy to treat, for the other types we rely on two separate methods: Type (2) covers are dealt with through a fundamental group argument supplemented by a \emph{coprimality trick}.  As for type (3) covers, we make use of the recently established \emph{weak Hilbert Property for abelian varieties} due to Corvaja, Demeio, Javanpeykar, Lombardo and Zannier \cite{CDJLZ}, which guarantees the existence of many rational points avoiding these covers. To control the pullbacks of rational curves on $\Kum(\Alb^1(H))$, we apply this result not to $\Alb^1(H)$ itself but instead to the base change of $\Alb^1(H)$ to the function field of $H$.

Finally, for any given finite collection of these covers, we need to carefully match these methods together, so as to extract a subfamily of rational curves which simultaneously have no rational sections along any one of them. Thus \emph{a fortiori} these curves contribute to many rational points not coming from these covers, thereby achieving the Hilbert Property for $\Kum(\Alb^1(H))$. The deduction of the Hilbert Property for $\Kum(\Jac(H))$ in the case $g=2$ is based on the fact (due to Cassels and Flynn, see \Cref{desing}) that $\Kum(\Alb^1(H))$ and $\Kum(\Jac(H))$ are birational to each other over $K$.
However, in light of the lack of an analogue of \Cref{{desing}} for $g>2$, we cannot extend \Cref{higher} to the setting where $H$ has no $K$-rational Weierstrass point (or slightly weaker, when $H$ has no zero-cycle of degree $1$ fixed under $\tau_H$ -- the existence of such a zero-cycle is for example guaranteed if $f$ factors into polynomials of coprime degrees). It would be interesting to remove the assumption on the degree of $f$ in future investigations.

The article is structured as follows. After notation and background in \S\ref{sec:notation}, in \S\ref{sec:curves} we construct and prove properties of rational curves on $\Kum(\Jac(H))$ and $\Kum(\Alb^1(H))$. In \S\ref{sec:pullback} the crucial Pullback Lemmas are formulated for covers of types (1)-(3). These are applied in \S\ref{sec:proofs} to prove the main theorems.

\section{Notation, conventions and preliminaries}\label{sec:notation}

A \emph{variety} over $K$ is a separated reduced $K$-scheme of dimension $>0$ and of finite type. To simplify exposition, throughout the rest of this article all varieties are supposed to be integral. We write $K(X)$ for the function field of a $K$-variety $X$.

Let $\phi:Y\to X$ be a morphism of varieties and let $X'\subseteq X$ be a subvariety. We say that $X'$ \emph{has a rational section along $\phi$}, if the morphism $\phi|_{\phi^{-1}(X')}:\phi^{-1}(X')\to X'$ has a rational section, i.e.\ there exists an open dense subset $U^\prime\subset X^\prime$ and $\psi:U^\prime\to Y$ such that $\phi\circ\psi$ is the identity map on $U^\prime$.

A $K$-variety birationally equivalent to $\BP^n_K$ is called $K$-\emph{rational}. By a \emph{rational curve} on a $K$-variety $X$, we mean a closed (possibly singular) subvariety on $X$ which is birationally equivalent to $\BP^1_K$. By a \emph{hyperelliptic curve} over $K$, we mean a smooth, projective curve over $K$ with a finite morphism to $\BP^1_K$ of degree $2$. (Note that a curve may be geometrically hyperelliptic but not hyperelliptic over $K$, the image of the canonical map being a conic without a $K$-rational point. This does not happen when $g$ is even.)

Assume that a $K$-variety $V$ is equipped with an involution $\tau:V\to V$, usually taken to be implicit from the context. Then we define the \emph{Kummer quotient} $\Kum(V)$ as the (possibly singular) quotient of $V$ by the $\BZ/2\BZ$-action for which the group generator acts by $\tau$.

Let $(G,+)$ be an abelian group. A \emph{coset} $\CC$ of $G$ is of the form $g+H$ where $g\in G$ and $H\subset G$ is subgroup.

Our next proposition says that proving the Hilbert Property boils down to a specific family of maps.
\begin{proposition}\label{prop:thincoverequiv}
 Let $X$ be a quasi-projective variety over a field $K$. Then a subset $M\subset X(K)$ is thin if and only if 
there exists a finite collection $(\phi_i:Y_i\to X)$, where $Y_i$ are normal and geometrically integral varieties over $K$, and $\phi_i$ are finite surjective $K$-morphisms of degree $[K(Y_i):K(X)]>1$, such that $M\setminus\bigcup_{i=1}^r \phi_i(Y_i(K))$ is not Zariski dense.
\end{proposition}
\begin{proof}
Upon considering the morphism from the variety $Y$, as the disjoint union of $Y_i$ with a Zariski closed subset of $X$, to $X$, the ``if'' part follows just from definition.

We now prove the ``only if'' part. Namely, for every thin set $M\subset X(K)$, we want to construct such a finite collection $\phi_i:Y_i\to X$ as above such that $M\setminus (\bigcup_{i=1}^r \phi_i(Y_i(K)))$ is contained in a proper Zariski closed subset. Assume that $M\subset f(Z(K))$, where $Z$ is a $K$-variety and the morphism $f$ is generically finite and has no rational section. Upon considering the irreducible components of $Z$, it suffices to assume that 
 $M=\phi(Y(K))$, where $Y$ is an integral $K$-variety with $\dim Y=\dim X$, and $\phi:Y\to X$ is a dominant $K$-morphism with degree $[K(Y):K(X)]>1$.
Let us fix a projective model $\overline{X}$ of $X$ in what follows.
 
 By the Nagata compactification theorem (see \cite[Theorem 3.2]{Lutkebohmert}), $\phi$ extends to a proper morphism $\overline{\phi}:\overline{Y}\to\overline{X}$ where $\overline{Y}$ is a proper model of $Y$. Using the Stein factorization (\cite[III. Corollary 11.5]{Hartshorne}), we can decompose $\overline{\phi}=g_1\circ g_2:\overline{Y}\to W\to \overline{X}$, where $g_1:\overline{Y}\to W$ is proper with connected fibres and $g_2:W\to \overline{X}$ is finite. It follows that $g_1$ is birational, and $g_2$ is proper by \cite[Exercise II 4.1]{Hartshorne}) and hence it is surjective (since $\phi$ is dominant). Let $Z$ denotes the non-normal locus of $W$ and let $h:\widehat{W}\to W$ be the normalisation, which is finite birational. We may assume that $\widehat{W}(K)\neq \varnothing$ and hence $\widehat{W}$ is geometrically integral (see \cite[p. 20]{Serre}). Otherwise $M$ is contained in the Zariski closed subset $g_2(Z)$.  We let $\widetilde{W}= \widehat{W}\setminus (h\circ g_2)^{-1}(\overline{X}\setminus X)$. We claim that the morphism $h\circ g_2|_{\widetilde{W}}:\widetilde{W}\to X$ meets our need. Indeed, $h\circ g_2$ is finite surjective of degree $[K(W):K(X)]=[K(Y):K(X)]>1$, and $M\setminus((h\circ g_2)(\widetilde{W}(K)))$ is contained in the Zariski closed subset $g_2(Z)$. 
\end{proof}
The proposition above motivates the following convention.

A \emph{cover} $\phi:Y\to X$ between $K$-varieties is a finite surjective morphism with $Y$ being normal and geometrically integral. The \emph{degree} of $\phi$ is $[K(Y):K(X)]$.

Therefore \Cref{prop:thincoverequiv} may be rephrased as follows: \begin{corollary}\label{co:thincoverequiv}
Let $X$ be a quasi-projective variety over a Hilbertian field $K$. Then $X$ has the Hilbert Property over $K$ if and only if for any finite family of covers $\phi_i\colon Y_i\to X$, $i=1,\dots,r$ with $\deg(\phi_i)>1$,  $X(K)\setminus\bigcup_{i=1}^r \phi_i(Y_i(K))$ is Zariski dense in $X$.\end{corollary}

\section{Rational curves on Jacobian Kummer varieties}\label{sec:curves}

Let $H$ be a hyperelliptic curve of genus $g$ over $K$. Recall the following classical facts about $H$, its Jacobian and Albanese varieties, and their Kummer quotients.

The hyperelliptic involution $\tau_H:H\to H$ is defined over $K$. The Kummer quotient $C=\Kum(H)$ is non-canonically isomorphic to $\BP^1_K$, and we write  $\pi_H:H\to C$ for the quotient morphism. The set $W\subset H(\ov K)$ of Weierstrass points of $H$, i.e.\ the fixed points under $\tau_H$, forms a zero-dimensional reduced $K$-subscheme of degree $2g+2$. 

\subsection{The Jacobian variety and rational curves on its Kummer quotient}\label{jacobian}
The Jacobian $J=\Jac(H)$ is a principally polarised abelian variety over $K$, even when $H(K)=\emptyset$. See e.g. \cite[Theorem 1.1, Summary 6.11]{Milne2}.  
For $n\in\BZ$, write $[n_J]:J\to J$ for the multiplication-by-$n$ morphism, and for a closed point $x\in J$, write $t_x:J\to J$ for the translation-by-$x$ morphism. The fixed points of the involution $\tau_J=[(-1)_J]$ consist precisely of the $2$-torsion $J[2]$, which is a zero-dimensional subgroup scheme of degree $2^{2g}$. Let $X=\Kum(J)$ and $\pi:J\to X$ be the quotient morphism. Clearly the morphisms $[n_J]$ and $\tau_J$ commute, so $[n_J]$ descends to $X$ and we denote the resulting morphism by $[n_X]$. The morphism $\pi$ is étale outside the images of $J[2]$, which form the singular locus of $X$.
Blowing up these $2^{2g}$ canonical singularities gives a minimal resolution $\tX\to X$ where $\tX$ is known as the \emph{Kummer variety} of $J$. It is a smooth, projective and simply connected variety (see \cite[Theorem 1]{Spanier}).

For every $x_0\in H(K)$, there exists an Abel--Jacobi closed $K$-embedding
\begin{align*}
     \iota_{x_0}: H&\hookrightarrow J,\\x &\mapsto\langle x_0-x\rangle, 
\end{align*}
 where $\langle \cdot\rangle $ denotes the class of a divisor. By the Riemann-Roch theorem, the closed points on $H$ generate $J$ (see \cite[IV Remark 4.10.9]{Hartshorne} or \cite[\S5]{Milne2}).

We assume throughout the rest of \S\ref{jacobian} that $H$ has a $K$-rational Weierstrass point $\infty$. 
We consider $H$ as a subvariety of $J$ via the Abel-Jacobi embedding $\iota_\infty$. Since the restriction of the map $\tau_J$ to $H$ is just $\tau_H$, the image of $H$ under $\pi$ is a smooth rational curve, which we denote by $C$. For every $n\in\BZ_{>0}$, let $H_n=[n_J](H)$, the image of $H$ under the map $[n_J]$, and let $C_n=\pi(H_n)$. The curves $C_n$ are rational for all $n\in\BZ_{>0}$. (See e.g. \cite[Chapter 13, Example 0.4]{Huybrechts}.)

\begin{lemma}\label{rational-curves}
The family $(H_n)_{n\in\BN}$ consists of pairwise distinct curves. Consequently, $(C_n)_{n\in\BN}$ are pairwise distinct rational curves on $X$. 
\end{lemma}
\begin{proof}
 Assume $H_m=H_n$ for some $0<m<n$. The $m^{2g}$ translates $t_x(H)$, $x\in J[m]$, are pairwise distinct. Since $[m_J]$ is étale of degree $m^{2g}$, it follows that $[m_J]|_H:H\to H_m$ is of degree $1$, i.e.\ a birational map. The same holds for $[n_J]|_H$. Thus there is an automorphism $\sigma$ of $H$ such that as rational maps
 \[\sigma=[m_J]|_H\circ([n_J]|_H)^{-1}=([n_J]|_H)^{-1}\circ[m_J]_H.\]
 Because $\sigma$ has finite order, after replacing $m$ and $n$ by powers thereof, we can assume $[m_J]_H=[n_J]_H$.
 It follows that $H\subset\Ker([(n-m)_J])\subset J$, which is impossible since the closed points of $H$ generate $J$.
\end{proof}

\begin{remark}
 We would like to point out an oversight in \cite[p. 274]{Huybrechts}, where it is stated that when $g=2$, the rational curve $C_n$ is smooth for all $n$. In fact this is false for $n>1$ and it is for this reason that we will need the coprimality argument in \Cref{prop:pullbackisogenyH}. If $C_n$ were smooth, then since it is an ample geometrically irreducible divisor of $J$, its pullback along $[n_J]$ would also be geometrically irreducible by \cite[Corollary 4.3.4]{Birkenhake-Lange}. But this is clearly not the case since it is (geometrically) a sum of translates of $C$. See also \Cref{rmk:pullbacklemma} below.
 
 More concretely, let us take a look at the curve $C_2$. Since $H$ contains $6$ points of $J[2]$, then $H_2=[2_J](H)$ passes through $0\in J$ with multiplicity at least $6$. Hence $C_2$ has multiplicity as least $3$ at $\pi(0)$. In general, the singular locus of $H_n$ (resp.\ $C_n$) is the closed subscheme of all points $[n_J]x$ (resp.\ $[n_X]\pi(x)$), $x\in H$, such that there exists $y\in H$ different from $x$ with $[n_J](x-y)=0$. In other words, $x$ is in a \emph{torsion package} of size $>1$ in the sense of \cite{Baker-Poonen}. We omit the proof since this fact will not be needed.
\end{remark}

\begin{lemma}\label{le:nondegenerate}
Let $\eta_H\in J_{K(H)}$ be the image of the generic point of $H$ under the embedding $\iota_\infty:H\hookrightarrow J$. 
Then $\eta_H$ is non-degenerate in $J_{K(H)}$, i.e.\ $\BZ \eta_H$ is Zariski dense in $J_{K(H)}$.
\end{lemma}
\begin{proof}
 Assume instead that $\eta_H$ is contained in a proper abelian subvariety of $J_{K(H)}$. Then by the Poincar\'e irreducibility theorem, there exists a non-zero endomorphism $e\in \End(J_{\ov {K(H)}})$ (actually in $\End(J_{K(H)})$) such that $e\eta_H=0$. Because $\End(J_{\ov {K(H)}})\simeq \End(J_{\ov K})$, see \cite[Lemma 1.2.1.2]{CCO}, it follows that there exists $e'\in \End(J_{\ov K})$ such that $e's(x)=e'x=0$ for all $x\in H$. Thus $H$ would have to be contained in a proper abelian subvariety of $J$, which is impossible again since the closed points of $H$ generate $J$.
\end{proof}
\begin{corollary}\label{le:densefiniteindexcoset}
    For every finite index coset $\CN\subset \BZ$, $\bigcup_{n\in \CN} H_n$ is Zariski dense in $J$ and $\bigcup_{n\in \CN} C_n$ is Zariski dense in $X$.
\end{corollary}
\begin{proof}
 Write $\CN=a\BZ+b$ where $a,b$ are integers with $a\neq 0$. Since the morphism $t_{b\eta_H}\circ [a_{J_{K(H)}}]:J_{K(H)}\to J_{K(H)}$ is dominant, the conclusion follows directly from \Cref{le:nondegenerate}.
\end{proof}

\subsection{The Albanese torsor $\Alb^1(H)$ and rational curves on its Kummer quotient}\label{se:alb}
The zero-cycles of degree $1$ (up to linear equivalence) on $H$ fixed under $\tau_H$ form a torsor $T$ under $J[2]$, which is trivial if $H$ has a $K$-rational Weierstrass point. Let $c\in \RH^1(K,J[2])$ be the associated cohomology class. Then for every fixed $\ov K$-Weierstrass point $\infty\in W$, $c$ is given by the cocycle $\gamma\mapsto \langle\gamma(\infty)-\infty\rangle$. The twist of $J$ by the image of $c$ in $\RH^1(K,J)$, i.e. the quotient of $J\times T$ by the diagonal action of $J[2]$, is isomorphic to the Albanese torsor $J^1=\Alb^1(H)$.

In general, we denote by $J^i=\Alb^i(H)$ the $i$-th Albanese torsor, which parametrises zero-cycles of degree $i$ on $H$. There is an isomorphism $J^i\simeq J^{i+2}(H)$ for all $i\geq 0$, given by adding $\langle 2\infty\rangle$, and $J^0\simeq J$. The variety $J^1$ is an initial object within the category of torsors $V$ under $J$ with a morphism $\iota_V:H\to V$ (see \cite[Appendix A]{Wittenberg}). The involution $\tau_H$ extends formally to an involution $\tau_{J^1}$ on $J^1$ in a way that is compatible with the involution $\tau_J$ on $J$ and the $J$-torsor structure. We write $\pi^1:J^1\to X^1=\Kum(J^1)$ for the corresponding Kummer quotient morphism.

In the case $g=2$, the surfaces $X$ and $X^1$ are projective duals of each other as quartic surfaces in $\BP^3$. (By \cite[Chapter 4]{Cassels-Flynn}, they are isomorphic if and only if at least one of the  translates of $C\subset X$ by $J[2](\ov K)$ -- the so-called \emph{tropes} -- is defined over $K$, which happens exactly when the sextic polynomial $f$ defining $H$ has a linear or cubic factor.)
However, it is pointed out in \cite[Chapter 16]{Cassels-Flynn} (somehow surprisingly) that 
\begin{theorem}[Cassels--Flynn]\label{desing}
 The minimal desingularisations of the surfaces $X$ and $X^1$ are isomorphic over $K$.
\end{theorem}  
\begin{proof}
We offer a more conceptional and simplified proof here. Indeed, consider the following involutions 
\begin{align*}
    \sigma_1: H\times H &\to H\times H\\ (x_1,x_2)&\mapsto (\tau_H(x_2),\tau_H(x_1)),\\
    \sigma_2: H\times H &\to H\times H\\ (x_1,x_2)&\mapsto (x_2,x_1).
\end{align*}
We define $\operatorname{Inv}^2(H)$ (resp. $\Sym^2(H)$) be the quotient of $H\times H$ by $\sigma_1$ (resp. by $\sigma_2$). 
For any point $(x_1,x_2)\in H\times H$, we write $\{x_1,x_2\}$ for the corresponding quotient image.

We now define
\begin{align*}
    \beta:\operatorname{Inv}^2(H) &\to J^0=J\\ \{x_1,x_2\}&\mapsto\langle x_1-x_2\rangle,\\
    \beta':\Sym^2(H)&\to J^2\\ \{x_1,x_2\}&\mapsto \langle x_1+x_2\rangle.
\end{align*} Then by the Riemann-Roch theorem  (see also \cite[Proposition 4(2)]{Baker-Poonen}), outside the image of the diagonal $(\id_H\times \id_H)(H)$ (resp. the anti-diagonal $(\id_H\times \tau_H)(H)$), $\beta$ (resp. $\beta'$) is an isomorphism. Hence both $\beta,\beta'$ are birational.

Let us define the rational maps $b,b'$ over $K$ by \begin{align*}
    b:X&\dashrightarrow X^1\\
    \{z,\tau_J(z)\}&\mapsto\{z+\langle x_2\rangle,z+\langle \tau_H(x_1)\rangle\}
\end{align*} for $z\in J\setminus J[2]$,
where $\{x_1,x_2\}$ is the preimage of $2z$ under $\beta$, and
\begin{align*}
    b': X^1 &\dashrightarrow X\\
\{z^1,\tau_{J^1}(z^1)\}&\mapsto\{z^1-\langle x\rangle,z^1-\langle x'\rangle\} 
\end{align*} for $z^1\in J^1\setminus T$,
where $\{x,x'\}$ is the preimage of $2z^1$ under $\beta'$. Then $b$ and $b'$ are inverse to each other, and hence are birational maps.  Blowing up the singular loci of $X$ and $X'$ gives the desired isomorphism.
\end{proof}
\begin{remark} Alternatively, in terms of cocycles, it can be checked that the image of $c$ in $\RH^1(K,\Aut_{\ov K}\ov K(X))$ is the coboundary of the automorphism of $\ov K(X)$ which is induced by the birational map \[\sigma:X_{\ov K}\dashrightarrow X_{\ov K}, \pi(z)\mapsto \pi(z+\langle x_2-\infty\rangle);\] thus $c$ vanishes in $\RH^1(K,\Aut_{\ov K}\ov K(X))$. 
\end{remark}

The image $H^1$ under the natural embedding $\iota_{J^1}:H\to J^1$ is stable under the involution $\tau_{J^1}$ and thus yields a smooth rational subcurve $C^1=\pi^1(H^1)\subset X^1$. For any \emph{odd} integer $n\in\BZ_{>0}$, we denote by $[n_{J^1}]$ the composition of the formal multiplication-by-$n$ map (of zero-cycles of degree $1$) 
$J^1\to J^n$ and the previously defined isomorphism $J^n\simeq J^1$. Because $[n_{J^1}]$ commutes with involution, it descends to a morphism $[n_{X^1}]$. We let $H^1_n=[n_{J^1}](H^1)$, the image of $H^1$ under the map $[n_{J^1}]$, and let $C^1_n=\pi^1(H^1_n)$. Since $$C^1_n=(\pi^1\circ [n_{J^1}])(H^1)=([n_{X^1}]\circ\pi^1)(H^1)=[n_{X^1}](C^1),$$  all the curves $C^1_n$ are rational.

If $H$ has a $K$-rational Weierstrass point $\infty$, then the $K$-morphism \begin{align*}
    \alpha_J:J^1&\to J\\ x&\mapsto x-\langle\infty\rangle
\end{align*} commutes with the involutions $\tau_J,\tau_{J^1}$, so it descends to a $K$-morphism $\alpha_X:X^1\to X$ and yields a commutative diagram
\[\begin{tikzcd}
    J^1\arrow[r,"\pi^1"]\arrow[d,"\alpha_J"] & X^1\arrow[d,"\alpha_X"]\\
    J \arrow[r, "\pi"] & X
   \end{tikzcd}\tag{$\dagger$}\]
where the vertical arrows are $K$-isomorphisms.
It is easy to see that $H^1_n$ maps to $H_n$ (as a subvariety of $J$) under $\alpha_J$, and $C_n^1$ maps to $C_n$ under $\alpha_X$ for all odd $n\in\BZ_{>0}$.

\section{Pullbacks along covers}\label{sec:pullback}
Throughout this section, $K$ denotes a finitely generated field of characteristic zero and $X=\Kum(J)$  denotes the Kummer quotient associated to the Jacobian $J$ of a genus $g$ hyperelliptic curve $H$ over $K$ with a $K$-rational Weierstrass point $\infty$.
We continue to consider $H$ embedded into $J$ via $\iota_\infty$.
\subsection{Pullback along covers factoring through the Jacobian}

\begin{lemma}\label{le:pullbacknongeomint}
    Let $\phi:Y\to X$ be a cover of degree $>1$. Suppose that the base change variety $Y\times_X J$ is not geometrically integral. Then for every finite index coset $\CN\subset\BZ$, there exists a subcoset $\CN^\prime\subset \CN$ of finite index, such that none of the rational curves $(C_n)_{n\in\CN^\prime}$ have a rational section along $\phi$ over $\overline{K}$.
\end{lemma}
\begin{proof}
 We may base change everything to $\overline{K}$ and we fix an algebraic closure $\Omega$ of $\overline{K}(X)$. We have that $Y_{\overline{K}}\times_{X_{\overline{K}}} J_{\overline{K}}$ is not geometrically integral if and only if $\ov K(Y)\otimes_{\ov K(X)}\ov K(J)$ is not a field, or equivalently, the field extensions $\ov K(Y)$ and $\ov K(J)$ in $\Omega$ are not linearly disjoint over $\ov K(X)$. Since $[\ov K(Y)\cdot \ov K(J):\ov K(Y)]\leq [\ov K(J):\ov K(X)]=2$, this happens if and only if $[\ov K(Y)\cdot \ov K(J):\ov K(Y)]<2$, or equivalently $\ov K(J)\subseteq \ov K(Y)$. It follows that, over $\overline{K}$, $\phi:Y\to X$ factors as a rational map $\phi_1:Y\dashrightarrow J$ then to $X$. 
 
 By \Cref{le:nondegenerate}, we fix a closed point $P\in H$ which is non-degenerate in $J$. By \cite[Lemma 4.6]{CDJLZ}, there exists a finite index subcoset $\CN^\prime\subset \CN$ such that the set $\{nP,n\in\CN^\prime\}$ does not intersect with the indeterminacy locus, say $Z\subset J$, of $\phi_1$. This implies that none of the curves $(H_n)_{n\in\CN^\prime}$ is contained in $Z$. If a rational curve $(C_n)_{\overline{K}}\subset X_{\overline{K}}$, $n\in\CN'$, has a rational section along $\phi$, say $\widetilde{\phi}$, then it has a rational section $\phi_1\circ\widetilde{
 \phi}$ along $\pi$. This is impossible since there are no non-constant rational maps from the projective line to abelian varieties (see \cite[Corollary~3.8]{Milne1}).
\end{proof}

\subsection{Pullbacks along unramified covers}

\begin{lemma}\label{le:surjfundgrp}
Let $D$ be a smooth, projective curve of genus $\geq 1$ over a separably closed field. Fix a point $x$ in $D$ inducing the Abel--Jacobi embedding $\iota_x:D\hookrightarrow \Jac(D)$. Let $\pi_1(D,x)$ (resp. $\pi_1(\Jac(D),0)$) be the pointed étale fundamental group of $D$ (resp. $\Jac(D)$). Then the map \[(\iota_x)_*:\pi_1(D,x)\to \pi_1(\Jac(D),0)\] induced by $\iota_x$ is surjective. Consequently, every pullback via $\iota_x$ of a connected finite \'etale covering of $\Jac(D)$ remains connected as a finite \'etale covering over $D$.
\end{lemma}
\begin{proof}
This is explained in \cite[Proposition 9.1]{Milne2}.
Indeed, $(\iota_x)_*$ induces an isomorphism \[\pi_1(D, x)^{\operatorname{ab}}\to\pi_1(\Jac(D),0),\] where $\pi_1(D,x)^{\operatorname{ab}}$ is the maximal abelian quotient of $\pi_1(D,x)$ and in particular $(\iota_x)_*$ is surjective. Thus by the theory of Galois categories, all connected finite \'etale coverings of $\Jac(D)$ have abelian Galois group, and the surjectivity of $(\iota_x)_*$ implies that all connected finite abelian coverings of $D$ are obtained precisely by pulling back the connected finite \'etale coverings of $J$. See also \cite[Proposition 5.5.4 (2)]{Szamuely}.
\end{proof}

\begin{remark}\label{rmk:pullbacklemma}
 If $g=2$, one may also prove \Cref{le:surjfundgrp} by using the fact that the pullback of a geometrically irreducible, unibranch, ample divisor on an abelian variety along an isogeny is again geometrically irreducible. (See \cite[Corollary 4.3.4]{Birkenhake-Lange}. Note that the statement is missing the crucial unibranch (or at the least smoothness) assumption.)
\end{remark}

\Cref{le:surjfundgrp} applies to $H\hookrightarrow J$, while an analogous statement for $H_n$, $n>1$ fails because of the non-smoothness. However, one can remedy the situation by restricting to a subfamily:

\begin{proposition}[``Coprimality trick'']\label{prop:pullbackisogenyH}
 Let $\phi:J'\to J$ be an unramified cover. Then $\phi^{-1}(H_n)$ are geometrically irreducible for all $n$ coprime to $\deg(\phi)$. 
\end{proposition}
\begin{proof}
We may base change everything to $\ov K$, without affecting the geometrical irreducibility. By the Lang--Serre theorem \cite[Chapter I \S2]{SerreAGCF}, $J'$ is an abelian variety and $\phi$ can be made into an isogeny. The case where $\deg(\phi)=1$ is clear and so we assume from now on that $\deg(\phi)>1$.
 
 If $n=1$, then the claim follows from \Cref{le:surjfundgrp}. In particular, $\Ker(\phi)\subset \phi^{-1}(H_1)$. In general, if $n$ is coprime to $\deg(\phi)=\deg(\Ker(\phi))$, then $[n_{J^\prime}]|_{\Ker\phi}$ is an isomorphism. Thus
 \begin{align*}
    \phi^{-1}(H_n)&=\phi^{-1}([n_J](H_1))\\ &=[n_{J^\prime}](\phi^{-1}(H_1))+\Ker(\phi)\\ &=[n_{J^\prime}]\left(\phi^{-1}(H_1)+\Ker(\phi)\right)\\ &=[n_{J^\prime}](\phi^{-1}(H_1)),
 \end{align*}
 and hence 
$\phi^{-1}(H_n)$ is geometrically irreducible.
\end{proof}

\begin{corollary}\label{co:pullbackrationalcurve}
    Let $\phi_i:Y_i\to X,i\in I$ be any finite collection of covers of with $\deg(\phi_i)>1$ such that $Y_i\times_X J$  are all geometrically integral, and the base change morphisms $\widetilde{\phi}_i:Y_i\times_X J\to J$ are all unramified. Then $\phi_i^{-1}(C_n)$ are all geometrically irreducible for all $n$ coprime to $\prod_{i}\deg(\phi_i)$. 
\end{corollary}
\begin{proof}
Since $H_n$ is stable under the involution,  $\phi_i^{-1}(C_n)$ is geometrically irreducible if and only if $\widetilde{\phi}_i^{-1}(H_n)$ is.  Since $Y_i\times_X J$ are all geometrically integral, we have $\deg(\widetilde{\phi}_i)=\deg(\phi_i)$. We then apply \Cref{prop:pullbackisogenyH} to $\widetilde{\phi}_i:Y_i\times_X J\to J$.
\end{proof}

\subsection{Pullbacks along ramified covers}
The following result is essentially \cite[Theorem 1.3]{CDJLZ}.
\begin{theorem}[Corvaja--Demeio--Javanpeykar--Lombardo--Zannier]\label{thm:CDJLZ1}
 Let $A$ be an abelian variety over a finitely generated field $K$ of characteristic $0$ with a non-degenerate point $x\in A(K)$. Let $f_i:B_i\to A$ be a finite collection of ramified covers. Then for any finite index coset $\CP\subset \BZ x$, there exists a finite index subcoset of $\CP^\prime\subset\CP$ such that for every $x'\in \CP^\prime$, $(f_i^{-1}(x'))(K)=\emptyset$.
\end{theorem}
\begin{proof}
 It suffices to work with a single ramified cover $f:B\to A$ by successively analysing the finite index coset obtained at each step. 
Write $\CP=\Pi+bx$ where $b\in\BZ$ and $\Pi$ is a finite index subgroup of $\BZ x$ and $P\in A(K)$. Define $f^\prime=t_{-bx}\circ f$ where $t_{-bx}:A\to A$ is the translation-by-$(-bx)$ morphism. Then we apply \cite[Theorem 1.3]{CDJLZ} to $\Pi$ and get a finite index coset $\Pi^\prime\subset\Pi$, for which any $K$-fibre of $f^\prime$ has no $K$-point. Then we may take $\CP^\prime=\Pi^\prime+bx$ to obtain the same  conclusion for $f$. 
\end{proof}

\begin{corollary}\label{pullbackramified}
 Let $\phi_i:Y_i\to X,i\in I$ be any finite collection of covers such that all the base change morphisms $\widetilde{\phi}_i:Y_i\times_X J\to J$ are ramified with $Y_i\times_X J$ geometrically integral. Then for any finite index coset $\CN\subset\BZ$, there exists a finite index subcoset $\CN'\subset\CN$ such that for any $n\in\CN'$, the rational curve $C_n$ does not have a rational section along $\phi_i$ for all $i\in I$.
\end{corollary}
\begin{proof}[Proof of \Cref{pullbackramified}]
 Consider the base change \[\widetilde{\phi}_{i,K(H)}:(Y_i\times_X J)_{K(H)}\to J_{K(H)}\] to the function field $K(H)$ of $H$. This is still a ramified morphism with $(Y_i\times_X J)_{K(H)}$ being geometrically integral.

 By \Cref{le:nondegenerate}, the $K(H)$-rational point $\eta_H\in J_{K(H)}$ induced by $\iota_\infty:H\hookrightarrow J$ is non-degenerate. By \Cref{thm:CDJLZ1} applied to $\CN \eta_H$ over the finitely generated field $K(H)$, there exists a finite index coset $\CN'\subset \CN$ such that for any $\eta'_H\in \CN' \eta_H\subset J_{K(H)}$,
 \[\eta'_H\notin \bigcup_{i\in I}\widetilde{\phi}_{i,K(H)}((Y_i\times_X J)_{K(H)}(K(H))).\]
 
 Consider the cartesian diagram
 \[\begin{tikzcd}
    (Y_i\times_X J)_{K(H)}\arrow[r]\arrow[d,"\widetilde{\phi}_{i,K(H)}"] & Y_{i,K(H)}\arrow[d,"\phi_{i,K(H)}"]\arrow[r] & Y_{i,K(C)}\arrow[d,"\phi_{i,K(C)}"]\\
    J_{K(H)} \arrow[r] & X_{K(H)} \arrow[r] & X_{K(C)}
   \end{tikzcd}
\] where the rightmost square is the base change from $K(C)$ to $K(H)$. Let $\eta_C\in X_{K(C)}$ be the image of $\eta_H$ under $J_{K(H)}\to X_{K(C)}$. 
 Then for any $\eta'_C\in\CN' \eta_C\subset X_{K(C)}$,
 \[\eta'_C\notin \bigcup_{i\in I}\phi_{i,K(C)}(Y_{i,K(C)}(K(C))).\]
 Equivalently, none of the curves $C_n$, $n\in\CN'$, has a rational section along $\phi_i$ for all $i\in I$.
\end{proof}

\section{Hilbert Property for Jacobian Kummer varieties}\label{sec:proofs}
\begin{proposition}\label{criterion}
 Let $X$ be a quasi-projective variety over a Hilbertian field $K$. Assume that there exists an infinite set $\CC$ of (not necessarily smooth) $K$-rational subvarieties on $X$ satisfying the following condition:
 
 For every finite family of covers $\phi_i:Y_i\to X$, $i=1,\dots, r$, with $\deg(\phi_i)>1$, there exists a subset $\CC'\subset \CC$ such that $\bigcup_{C\in\CC'}C$ is Zariski dense in $X$ and for all $C\in\CC'$ and $i=1,\dots, r$, $C$ does not have a rational section along $\phi_i$.
 
 Then $X$ has the Hilbert Property.
\end{proposition}
\begin{proof}
 Let us fix a finite family of covers $\phi_i:Y_i\to X$, $i=1,\dots, r$  with $\deg(\phi_i)>1$. 
Let $\CC'\subset\CC$ be as stated in the proposition. Because the union of $C\in\CC'$ is Zariski dense, upon replacing $\CC'$ by a subset if necessary, we may assume that for every subvariety $C\in\CC'$, $C$ is not contained in the ramification locus of any $\phi_i$ for $i=1,\dots,r$.
 In view of \Cref{prop:thincoverequiv}, to show the Hilbert Property for $X$, it suffices to show that 
 $$\left(\bigcup_{C\in \CC'}C(K)\right)\setminus \left(\bigcup_{i=1}^r\phi_i(Y_i(K))\right)$$ is not contained in any proper Zariski closed subset of $X$.
 
 Assume that there exists $Z\subset X$ proper Zariski closed containing all the $K$-points above. We fix $C_0\in\CC'$ such that $C_0\not\subset Z$. For every $i\in\{1,\dots, r\}$, let ${C_{ij}}$, $j=1,\dots,k_i$, be the irreducible components of $\phi_i^{-1}(C_0)$. 
 The maps $\phi_i|_{C_{ij}}:C_{ij}\to C_0$ are finite morphisms because $C_0$ does not have a rational section along $\phi_i$, and since $C_0$ is not contained in the ramification locus of $\phi_i$ by assumption, if $\dim(C_{ij})=\dim(C_0)$ then $\phi_i|_{C_{ij}}$ is dominant, and hence has degree $>1$. It thus follows from the Hilbert Property for rational varieties (see \cite[p. 30 Examples]{Serre}) that $C_0(K)$ is not thin, whence $$C_0(K)\setminus \left((Z\cap C_0)(K)\bigcup\left(\bigcup_{i=1}^r\bigcup_{j=1}^{k_i} \phi_i(C_{ij}(K))\right)\right)\neq\varnothing,$$ which is a contradiction. This finishes the proof.
\end{proof}

We now turn to the proofs of the main theorems. We will apply \Cref{criterion} only in the case where $\CC$ is a countable set of subcurves.

\subsection*{Proof of \Cref{alb}}
Consider the rational $K$-subcurves $\{C^1_n\}_{n\in1+2\BZ}$ on $X^1$. Let $\phi_i:Y_i\to X^1$, $i=1,\dots, r$ be any fixed finite family of covers with $\deg(\phi_i)>1$. We want to construct a finite index subcoset $\CN'\subset 1+2\BZ$ of odd integers such that $\{C^1_n\}_{n\in\CN'}$ such that all conditions of \Cref{criterion} are satisfied. Let us take a finite field extension $L$ of $K$ such that $H$ has an $L$-rational Weierstrass point $\infty$. We base change everything to $L$ (and we omit the subscript $L$). In view of the diagram $(\dagger)$, we shall replace $X^1$ by the isomorphic $X$ and $C^1_n$ by $C_n$ in the following. 

Let $\wt\phi_i:\wt Y_i\to J$ be the base change of $\phi_i$ along the Kummer quotient $\pi:J\to X$. It may be assumed that for certain $r_1$, all $\wt Y_i$, $1\leq i< r_1$ are not geometrically integral and the remaining $\wt Y_i$, $r_1\leq i\leq r$ are geometrically integral. For $r_1\leq i\leq r$, we may moreover assume that for certain $r_2\geq r_1$, the morphisms $\wt\phi_i,r_1\leq i\leq r_2$ are all unramified, while the remaining $\wt\phi_i,r_1<i\leq r$ are all ramified. 

We start with $$\CC=\{C_{n} : n \in 1+2\BZ\}.$$ If $r_2>r_1$, i.e.\ certain $\wt\phi_i$ are unramified, we consider $$d=\lcm(\deg\phi_i,i=r_1,\dots,r_2), \quad\CN^{\acute{e}t}=1+2d\BZ.$$ Then $\CN^{\acute{e}t}$ is a finite index coset of $\BZ$, and by \Cref{co:pullbackrationalcurve}, any $C_n$ with $n\in\CN^{\acute{e}t}$ does not have a rational section along $\phi_i$ for all $r_1\leq i\leq r_2$.
Otherwise if there are no such unramified covers, we just take $\CN^{\acute{e}t}=1+2\BZ$.

By \Cref{pullbackramified}, one obtains a finite index subcoset $\CN^{geo}\subset\CN^{\acute{e}t}$ such that any $C_n$ with $n\in\CN^{geo}$ does not have a rational section along $\phi_i$ for all $i=r_2+1,\dots,r$. Then on applying \Cref{le:pullbacknongeomint}, one obtains a further finite index subcoset $\CN'\subset\CN^{geo}$ such that any $C_n$ with $n\in\CN'$ does not have a rational section along $\phi_i$ for all $i=1,\dots,r_1-1$ (still over the field extension $L$ of $K$).

Any such $C_n^1,n\in \CN'$ cannot have a rational section along any $\phi_i$ over the base field $K$ as well. Thus, by using again \Cref{le:densefiniteindexcoset}, all conditions of \Cref{criterion} are satisfied for the family of rational curves $$\CC'=\{C_n:n\in\CN'\},$$ from which we conclude the Hilbert Property.
\qed

\subsection*{Proof of \Cref{dim2}} This follows from \Cref{alb} together with \Cref{desing} using the birational invariance of the Hilbert Property. \qed

\subsection*{Proof of \Cref{higher}} This follows immediately from \Cref{alb} since $H$ has by assumption a $K$-rational Weierstrass point (at infinity), and thus $X\cong X^1$ over $K$ in view of the diagram $(\dagger)$.  \qed

\section*{Acknowledgements} We are in debt to J.\ Demeio for generously sharing many key insights into the Hilbert Property which simplify and generalise our previous arguments considerably. The first author was supported by the Engineering and Physical Sciences Research Council (Grant number EP/R513143/1) and UKRI award UKRI094. The second author was supported by the DFG grant DE1646/4-2 at Leibniz Universität Hannover, and by the FWF grant No. P32428-N35 at IST Austria. This work was partly undertaken at the workshop ``Arithmetic Statistics and Local-Global Principles'' hosted by the Erwin Schrödinger International Institute for Mathematics and Physics, Vienna and the workshop ``Rational Points on Higher-Dimensional Varieties'' hosted by the International Centre for Mathematical Sciences, Edinburgh. We thank the organisers and the institutes for the hospitality.
Last but not least, we thank the anonymous referee for helpful remarks. 

\printbibliography

\end{document}